\documentclass[a4paper,10pt,twoside,reqno,nonamelimits]{article}
\usepackage{a4wide}
\usepackage[colorlinks=true,urlcolor=blue,linkcolor=blue,citecolor=blue]{hyperref}
\usepackage{newcent}         % PSNFSS

\usepackage{amsmath}
\usepackage{amssymb}
\usepackage{amsfonts}
\usepackage{amscd}
\usepackage{amsthm}
\usepackage{amsxtra}
\usepackage{mathrsfs}
\usepackage{nicefrac}
\usepackage{bbold}
\usepackage{graphicx}%\graphicspath{ {image/} }
\usepackage{xcolor}
\usepackage{tikz}
\usepackage{enumitem}
\setlist[enumerate]{leftmargin=2em,itemindent=0em, labelindent=0pt,labelwidth=1.5em,labelsep=.5em, align=left, noitemsep}
\newlist{txtenum}{enumerate}{1}
\setlist[txtenum]{leftmargin=0em,itemindent=1.5em, labelindent=0pt,labelwidth=1em,labelsep=.5em, align=left}
\usepackage[norelsize,boxed,noend,linesnumbered]{algorithm2e}\DontPrintSemicolon
\SetKwFor{RepeatTimes}{repeat}{times}{endrepeat}
\RestyleAlgo{ruled}
\theoremstyle{plain}
\newtheorem{theorem}{Theorem}
\newtheorem*{theorem*}{Theorem}

\newtheorem{proposition}[theorem]{Proposition}
\newtheorem*{proposition*}{Proposition}

\newtheorem{corollary}[theorem]{Corollary}
\newtheorem*{corollary*}{Corollary}

\newtheorem{lemma}[theorem]{Lemma}
\newtheorem*{lemma*}{Lemma}

\newtheorem*{observation*}{Observation}

\newtheorem*{conjecture*}{Conjecture}

\newtheorem*{question*}{Question}

\newtheorem*{questions*}{Questions}

\newtheorem*{problem*}{Problem}

\newtheorem*{problems*}{Problems}

\newtheorem*{openproblem*}{Open Problem}

\theoremstyle{definition}

\newtheorem*{definition*}{Definition}

\newtheorem*{example*}{Example}

\newtheorem*{exercise*}{Exercise}

\newtheorem*{remark*}{Remark}

\newtheorem*{remarks*}{Remarks}

\theoremstyle{remark}

\newtheorem*{claim*}{Claim}

\newcommand{\subclass}[1]{}

\newcommand{\enumTi}[1]{\renewcommand{\theenumi}{#1}}

\newcommand{\alphenumi}{\enumTi{\alph{enumi}}}

\newcommand{\romenumi}{\enumTi{\roman{enumi}}}

\newlength{\hspaceforlengthglumpf}

\renewcommand{\em}{\sl}

\newcommand{\lt}{\left}
\newcommand{\rt}{\right}

\newcommand{\abs}[1]{{\lt\lvert{#1}\rt\rvert}}

\newcommand{\NN}{\mathbb{N}}

\newcommand{\RR}{\mathbb{R}}

\newcommand{\ZZ}{\mathbb{Z}}

\newcommand{\DDelta}{\mathbb{\Delta}}

\DeclareMathOperator*{\Prb}{\mathbb{P}}

\newcommand{\Mtx}[1]{{\bigl(\begin{smallmatrix}#1\end{smallmatrix}\bigr)}}

\newlength{\algotabbingwidth}
\newcommand{\RV}[1]{\mathbf{#1}}
\DeclareMathOperator{\Rg}{Rg}
\DeclareMathOperator{\MI}{MI}
\DeclareMathOperator{\CoI}{CoI}
\DeclareMathOperator{\SI}{SI}
\DeclareMathOperator{\CI}{CI}
\DeclareMathOperator{\UIy}{UIy}
\DeclareMathOperator{\UIz}{UIz}
\DeclareMathOperator{\Dom}{dom}
\newcommand{\SIext}{\SI^\text{ext}}
\newcommand{\SIprext}{\SI^\text{prext}}
\newcommand{\mySI}{\SI^\clubsuit}
\newcommand{\subgrad}{\mathfrak{g}}
\begin{document}
\title{Optimizing Bivariate Partial Information Decomposition}
\author{Abdullah Makkeh, Dirk Oliver Theis\thanks{Supported by the Estonian Research Council, ETAG (\textit{Eesti Teadusagentuur}), through PUT Exploratory Grant \#620, and by the European Regional Development Fund through the Estonian Center of Excellence in Computer Science, EXCS.}\\[1ex]
  \small Institute of Computer Science {\tiny of the } University of Tartu\\
  \small \"Ulikooli 17, 51014 Tartu, Estonia\\
  \small \texttt{\{makkeh,dotheis\}@ut.ee}%
}
%%%%%%%%%%%%%%%%%%%%%%%%%%%%%%%%%%%%%%%%%%%%%%%%%%%%%%%%%%%%%%%%%%%%%%%%%%%%%%%%%%%%%%%%%%%%%%%%%%%%%%%%%%%%%%%%%%%%%%%%%%%%%%%%%%%%%%%%%%%%%%%%%%%%%% 
%%
%%
\date{\today}
%%
%%
%%%%%%%%%%%%%%%%%%%%%%%%%%%%%%%%%%%%%%%%%%%%%%%%%%%%%%%%%%%%%%%%%%%%%%%%%%%%%%%%%%%%%%%%%%%%%%%%%%%%%%%%%%%%%%%%%%%%%%%%%%%%%%%%%%%%%%%%%%%%%%%%%%%%%% 
\maketitle

\begin{abstract}
  None of the BROJA information decomposition measures $\SI, \CI, \UIy, \UIz$ are convex or concave over the probability simplex. In this paper, we provide formulas for the sub-gradient and super-gradients of any of the information decomposition measures. Then we apply these results to obtain an optimum of some of these information decomposition measures when optimized over a constrained set of probability distributions.
  %\par\medskip%
  %\textbf{Keywords:} Partial Information Decomposition, Non-convex Optimization, First Order Methods.
\end{abstract}

\section{Introduction}\label{sec:intro}
\subsubsection*{Terminology and notation}
We use the common shorthand $[n] := \{1,\dots,n\}$. For vectors, we use the following summation convention: Replacing an index by an asterisk $*$ has the effect summing over all the possible values, e.g., for $p\in\RR^{A\times B\times C}$, the term $p_{a,*,c}$ stands for $\lt(\sum_{b\in B} p_{a,b,c}\rt)$, e.g.,
\begin{equation*}
  p_{a,*,c} \; p_{*,b,c} = \lt(\sum_{b\in B} p_{a,b,c}\rt) \lt(\sum_{a\in A} p_{a,b,c}\rt)
\end{equation*}

All random variables considered in this paper have finite range (unless explicitly stated otherwise).  Denote by $\Rg\RV{X}$ the range\footnote{The range is a set with the property $\Prb(\RV{X}=x)>0$ for all $x\in\Rg\RV{X}$, and $\Prb(\RV{X}=x)=0$ for all $x\not\in\Rg\RV{X}$.  If a range exists it is unique; if the range exists and is finite, we say that the random variable has ``finite range''.} of the (finite-range) random variable~$\RV{X}$.

For a (finite) set~$X$, we denote the \textit{probability simplex} by
\begin{equation*}
  \DDelta^X := \{ p \in \RR_+^X \mid p_* = 1 \}
\end{equation*}
For us, a probability distribution on a set~$X$, is a vector in~$\DDelta^X$.

\section{Main Theorem: Derivatives of PID-Quantities}

\begin{subequations}\label{eq:BROJA-CP}
  \begin{align}
    M(p) :=           &\max h(q) \\
    \text{over }      &q\in\RR^{S\times Y\times Z} \\
    \text{subject to }&q_{s,y,*} = p_{s,y,*} \qquad\text{for all $(s,y)\in S\times Y$;}\label{eq:BROJA-CP-1}\\
    {}                &q_{s,*,z} = p_{s,*,z} \qquad\text{for all $(s,z)\in S\times Z$;}\label{eq:BROJA-CP-2}\\
    {}                &q_{s,y,z} \ge 0       \qquad\text{for all $(s,y,z)\in S\times Y \times Z$.}\label{eq:BROJA-CP-3}
  \end{align}
\end{subequations}

\begin{proposition}[Corollary~3 in \cite{makkeh2017bivariate}]\label{prop:slackness}
  A feasible point~$q$ is an optimal solution to~\eqref{eq:BROJA-CP}, if and only if there exist $\lambda \in \RR^{S\times Y}$ and $\mu\in\RR^{S\times Z}$ satisfying the following:
  \begin{enumerate}[label=(\alph*)]
  \item For all $(y,z)\in Y\times Z$ with $q_{*,y,z}>0$:
    \begin{equation*}
      \lambda_{s,y} + \mu_{s,z} = \ln\Bigl( \frac{ q_{s,y,z} }{ q_{*,y,z} } \Bigr) \qquad\text{holds for all $s\in S$;}
    \end{equation*}
  \item For all $(y,z)\in Y\times Z$ with $q_{*,y,z}=0$, there is a
    probability distribution $\varrho$ with support~$S$ such that
    \begin{equation*}
      \lambda_{s,y} + \mu_{s,z} \le \ln(\varrho^{y,z}_s) \qquad\text{holds for all $s\in S$.}
    \end{equation*}
  \end{enumerate}
\end{proposition}

If $q,\lambda,\mu$ are as in the proposition, then we say that $\lambda,\mu$ are Lagrange multiplyers certifying optimality.

\begin{lemma}\label{lem:subgrad}
    Suppose $p$ has full support.  Let~$q$ be an optimal solution of~\eqref{eq:BROJA-CP}, and let~$\lambda,\mu$ be Lagrange multipliers certifying optimality.
    \begin{enumerate}[label=(\alph*)]
    \item\label{lem:grad:lag} If $q_{s,y,z} > 0$ for all $(s,y,z)\in S\times Y \times Z$, then~$M$ is differentiable in~$p$, and we have
      \begin{equation}
        \partial_{s,y,z} M\Mtx{p} = -\lambda_{s,y} - \mu_{s,z}\label{eq:lag}
      \end{equation}
    \item In any case, the vector defined by
      \begin{equation}
        \subgrad(p)_{s,y,z} := -\lambda_{s,y} - \mu_{s,z}
      \end{equation}
      is a super-gradient on~$M$ in the point $p$.
    \end{enumerate}
\end{lemma}
\begin{proof}
  	If $q$ is the optimal solution of~\eqref{eq:BROJA-CP}, then $$M(p)=\max_{q} h(q) = -\min_{q} -h(q)=h(q)$$ where $h(q) = H(S\mid Y,Z)$. So, the gradient of $M$ in $p$ is
  	\begin{equation}\label{eq:lem:grad}
  	    \nabla M(p) = -\lt( \ln\lt(\frac{q_{s,y,z}}{q_{*,y,z}}\rt)\rt)_{s,y,z}
  	\end{equation}
	If $q_{s,y,z}>0$ for all $(s,y,z)\in S\times Y\times Z$, then $q_{*,y,s}>0$ for all $(y,s)\in Y\times Z$ and so $M$ is differentiable in $p$. Moreover, Equation~\eqref{eq:lag} follows from the fact that $q,\lambda,\mu$ are as in Proposition~\ref{prop:slackness} and the gradient defined in~\eqref{eq:lem:grad}. 
	
	From~\cite[Proposition 2]{makkeh2017bivariate} and Proposition~\ref{prop:slackness}, we have $\lambda_{s,y} + \mu_{s,z}$ is a sub-gradient to $M'(p) := \min_{q} -h(q)$ subject to the constraints~\eqref{eq:BROJA-CP-1},~\eqref{eq:BROJA-CP-2}, and~\eqref{eq:BROJA-CP-3} in the point $p$. Hence $-\lambda_{s,y} -\mu_{s,z}$ is a super-gradient on $M$ in the point $p$.
\end{proof}

We would like to emphasize that, in this lemma as well as in the following results, the condition that~$p$ has full support is only there to simplify notation, and can be readily abandoned.%\todo[Abed: make sure that the gradient is really 0 if a (y,z)-pair doesn't exist.]

\begin{lemma}[\cite{ruszczynski2006nonlinear}, Lemma 2.73]\label{lem:grad-dir-derivative}
    Let $f:\RR^n\to\RR$ be a convex function and $p\in\Dom(f)$. A vector $\subgrad$ is a subgradient of $f$ in the point $p$ iff 
    \begin{equation*}
        f'(p,d)\ge\subgrad^T d\quad\text{for all $d\in\RR^n.$}
    \end{equation*}
\end{lemma}
\begin{theorem}\label{thm:der-PID}
  Suppose $p$ has full support.  Let~$q$ be an optimal solution of~\eqref{eq:BROJA-CP}, and let~$\lambda,\mu$ be Lagrange multipliers certifying optimality.
  \begin{enumerate}[label=(\alph*)]
  \item\label{thm:PID-quan} If $q_{s,y,z} > 0$ for all $(s,y,z)\in S\times Y \times Z$, then $\CI$, $\SI$, $\UIy$, $\UIz$ are all differentiable in~$p$, and we have
    \begin{subequations}
      \begin{align}
        \partial_{s,y,z} \CI(p)  &= \ln\lt(\frac{p_{*,y,z}}{p_{s,y,z}}\rt) - \lambda_{s,y} - \mu_{s,z} \\
        \partial_{s,y,z} \SI(p)  &= - 1  + \ln\lt(\frac{p_{s,y,*} p_{s,*,z}}{p_{x,*,*}p_{*,y,*} p_{*,*,z}}\rt) - \lambda_{s,y} - \mu_{s,z} \\
        \partial_{s,y,z} \UIy(p) &= \ln\lt(\frac{p_{*,*,z}}{p_{s,*,z}}\rt) + \lambda_{s,y} + \mu_{s,z} \\
        \partial_{s,y,z} \UIz(p) &= \ln\lt(\frac{p_{*,y,*}}{p_{s,y,*}}\rt) + \lambda_{s,y} + \mu_{s,z}
      \end{align}
    \end{subequations}
  \item\label{thm:PID-super-gradient} In any case, the vectors defined by 
     \begin{subequations}
         \begin{align}
            \subgrad_{\CI}(p)_{s,y,z}     &= \ln\lt(\frac{p_{*,y,z}}{p_{s,y,z}}\rt) -\lambda_{s,y} -\mu_{s,z}\label{eq:sup-grad-CI}\\
            \subgrad_{\SI}(p)_{s,y,z}     &= - 1  + \ln\lt(\frac{p_{s,y,*} p_{s,*,z}}{p_{x,*,*}p_{*,y,*} p_{*,*,z}}\rt) - \lambda_{s,y} - \mu_{s,z}\label{eq:sup-grad-SI}
         \end{align}
     \end{subequations}
     are local super-gradients of $\CI$ and $\SI$ respectively and the vectors defined by 
     \begin{subequations}
         \begin{align}
            \subgrad_{\UIy}(p)_{s,y,z}    &= \ln\lt(\frac{p_{*,*,z}}{p_{s,*,z}}\rt) + \lambda_{s,y} + \mu_{s,z}\label{eq:sub-grad-UIy} \\
            \subgrad_{\UIz}(p)_{s,y,z}    &= \ln\lt(\frac{p_{*,y,*}}{p_{s,y,*}}\rt) + \lambda_{s,y} + \mu_{s,z}\label{eq:sub-grad-UIz}
         \end{align}
     \end{subequations}
     are local subgradients of $\UIy$ and $\UIz$ in the point $p$ respectively.
  \end{enumerate}
\end{theorem}
\begin{proof}
    For~\ref{thm:PID-quan}, Bertschinger et al. in~\cite{bertschinger-rauh-olbrich-jost-ay:quantify:2014} defined the partial information decomposition as follows:
    \begin{equation*}
      \begin{aligned}
        \CI(p)  &= \MI(S;Y,Z) - \min_{q} \MI(S;Y,Z) \\
        \SI(p)  &= \max_{q}\CoI(S;Y;Z) \\
        \UIy(p) &= \min_{q}\MI(S;Y\mid Z)\\
        \UIz(p) &= \min_{q}\MI(S;Z\mid Y)\\
      \end{aligned}
    \end{equation*}
    where the optimization is subject to the constraints~\eqref{eq:BROJA-CP-1},~\eqref{eq:BROJA-CP-2}, and~\eqref{eq:BROJA-CP-3}. Using the definition of $\MI(S;Y,Z)$ and the chain rule, we get
    % \begin{equation*}
    %     \begin{split}
    %         \CI(p)  &= \MI(S;Y,Z) - \min_{q} \MI(S;Y,Z) \\
    %                 &= \MI(S;Y,Z) - H(S) -\min_{q} -H(S\mid Y,Z) \\
    %                 &= -H(S\mid Y,Z) + \max_{q} H(S\mid Y,Z)\\
    %                 &= -H(S,Y,Z) + H(Y,Z) + M(p)\\
    %     \end{split}
    % \end{equation*}
    % \begin{equation*}
    %     \begin{split}
    %         \SI(p)  &= \max_{q}\CoI(S;Y;Z) \\
    %                 &= \MI(S;Y) + \MI(S;Y) + \max_{q} -\MI(S;Y,Z) \\
    %                 &= \MI(S;Y) + \MI(S;Z) -H(S) + \max_{q} H(S\mid Y,Z)\\
    %                 &= H(S) - H(S|Z) + H(S) - H(S\mid Z) -H(S) + M(p)\\
    %                 &= H(S) + H(Y) + H(Z) - H(S,Y) - H(S,Z) + M(p)\\
    %     \end{split}
    % \end{equation*}
    % \begin{equation*}
    %     \begin{split}
    %         \UIy(p) &= \min_{q}\MI(S;Y\mid Z) \\
    %                 &= -\MI(S;Z) + \min_{q} \MI(S;Y,Z) \\
    %                 &= -\MI(S;Z) + H(S) + \min_{q} -H(S\mid Y,Z)\\
    %                 &= H(S|Z) - \max_{q} H(S\mid Y,Z)\\
    %                 &= H(S,Z) - H(Z) - M(p)
    %     \end{split}
    % \end{equation*}
    % \begin{equation*}
    %     \begin{split}
    %         \UIz(p) &= \min_{q}\MI(S;Z\mid Y) \\
    %                 &= -\MI(S;Y) + \min_{q} \MI(S;Y,Z) \\
    %                 &= -\MI(S;Y) + H(S) + \min_{q} -H(S\mid Y,Z)\\
    %                 &= H(S|Y) - \max_{q} H(S\mid Y,Z)\\
    %                 &= H(S,Y) - H(Y) - M(p)
    %     \end{split}
    % \end{equation*}

    \begin{equation*}
      \begin{aligned}
        \CI(p)  &= M(p) - H(S\mid Y,Z) \\
        \SI(p)  &= M(p) + \MI(S;Y) - H(S\mid Z) \\
        \UIy(p) &= \MI(S;Z) + H(S) -M(p)\\
        \UIz(p) &= \MI(S;Y) + H(S) -M(p)\\
      \end{aligned}
    \end{equation*}
    where $H(S\mid Y,Z), H(S\mid Z),\MI(S;Y),$ and $\MI(S;Y)$ are functions of $p$.  By direct computations the equations in~\ref{thm:PID-quan} follow using the fact $\partial_{s,y,z} M(p) = -\lambda_{s,y} - \mu_{s,z}$. 
    
    % \begin{equation*}
    %     \begin{split}
    %         \partial_{s,y,z}\CI(p)  &= \partial_{s,y,z} -H(S,Y,Z) + H(Y,Z) + M(p)\\
    %                 &= -(- 1 - \ln(p_{s,y,z})) - 1 - \ln(p_{*,y,z}) - \lambda_{s,y} - \mu_{s,z}\\
    %                 &= \ln\lt(\frac{p_{*,y,z}}{p_{s,y,z}}\rt) - \lambda_{s,y} - \mu_{s,z}
    %     \end{split}
    % \end{equation*}
    
    % \begin{equation*}
    %     \begin{split}
    %         \partial_{s,y,z}\SI(p)  &= \partial_{s,y,z} H(S) + H(Y) + H(Z) - H(S,Y) - H(S,Z) + M(p)\\
    %                 &= - 1  - \ln(p_{s,*,*})  - 1  - \ln(p_{*,y,*}) - 1  - \ln(p_{*,*,z})  - (- 1  - \ln(p_{s,y,*}) - (- 1  - \ln(p_{s,*,z}) - \lambda_{s,y} - \mu_{s,z}\\
    %                 &= -1  + \ln\lt(\frac{p_{s,y,*} p_{s,*,z}}{p_{x,*,*}p_{*,y,*} p_{*,*,z}}\rt) - \lambda_{s,y} - \mu_{s,z}\\
    %     \end{split}
    % \end{equation*}
    % \begin{equation*}
    %     \begin{split}
    %         \partial_{s,y,z}\UIy(p) &= \partial H(S,Z) - H(Z) - M(p)\\
    %                 &= - 1 - \ln(p_{s,*,z}) - (- 1 - \ln(\ln(p_{*,*,z})) + \lambda_{s,y} + \mu_{s,z}\\
    %                 &= \ln\lt(\frac{p_{*,*,z}}{p_{s,*,z}}\rt) + \lambda_{s,y} + \mu_{s,z}\\
    %     \end{split}
    % \end{equation*}
    % \begin{equation*}
    %     \begin{split}
    %         \partial_{s,y,z}\UIy(p) &= \partial H(S,Y) - H(Y) - M(p)\\
    %                 &= - 1 - \ln(p_{s,y,*}) - (- 1 - \ln(\ln(p_{*,y,*})) + \lambda_{s,y} + \mu_{s,z}\\
    %                 &= \ln\lt(\frac{p_{*,y,*}}{p_{s,y,*}}\rt) + \lambda_{s,y} + \mu_{s,z}\\
    %     \end{split}
    % \end{equation*}
    For~\ref{thm:PID-super-gradient}, let 
    \begin{equation}\label{eq:g-funs}
        \begin{aligned}
            g_{\CI}(p)    &= H(S\mid Y,Z)         \\
            g_{\SI}(p)    &= \MI(S;Y) - H(S\mid Z) \\
            g_{\UIy}(p)   &= \MI(S;Z) + H(S)          \\
            g_{\UIz}(p)   &= \MI(S;Y) + H(S).\\  
        \end{aligned}
    \end{equation}
    Since $p$ has a full support then all the functions in~\eqref{eq:g-funs} are differentiable and
    \begin{equation}\label{eq:dir-g-funs}
        \begin{aligned}
            g'_{\CI}(p,d)    &= \sum_{s,y,z}\ln\lt(\frac{p_{*,y,z}}{p_{s,y,z}}\rt)d_{s,y,z} \\
            g'_{\SI}(p,d)    &= \sum_{s,y,z}\lt(\ln\lt(\frac{p_{s,y,*} p_{s,*,z}}{p_{x,*,*}p_{*,y,*} p_{*,*,z}}\rt)-1\rt)d_{s,y,z} \\
            g'_{\UIy}(p,d)   &= \sum_{s,y,z}\ln\lt(\frac{p_{*,*,z}}{p_{s,*,z}}\rt)d_{s,y,z} \\
            g'_{\UIz}(p,d)   &= \sum_{s,y,z}\ln\lt(\frac{p_{*,y,*}}{p_{s,y,*}}\rt)d_{s,y,z}.\\  
        \end{aligned}
    \end{equation}
     From Lemma~\ref{lem:subgrad} and Lemma~\ref{lem:grad-dir-derivative}, $\subgrad(p)$ is a super-gradient of $M$ at $p$ and  for any $d\in\RR^{S\times Y\times Z}$, we have   $-M'(p,d)\ge -\subgrad^T d$. Hence, the vectors defined by~\eqref{eq:sup-grad-CI} and~\eqref{eq:sup-grad-SI} are super-gradients of $\CI$ and $\SI$ respectively and the vectors defined by~\eqref{eq:sub-grad-UIy} and~\eqref{eq:sub-grad-UIz} are local subgradients of $\UIy$ and $\UIz$ in the point $p$ respectively.
\end{proof}

\begin{corollary}
  Let $I$ be any of $\CI,\SI$, $\UIy,\UIz$.  At the points where~$I$ is not smooth it is
  \begin{enumerate}[label=(\alph*)]
  \item concave%\todo[Abed: What does it mean to be convex/concave ``in a point''?!?]
  , in the case of $I=\CI,\SI$;
  \item convex, in the case of $I=\UIy,\UIz$.
  \end{enumerate}
\end{corollary}
\begin{proof}
  Using Theorem~\ref{thm:PID-quan}, the vectors $\subgrad_{\CI}(p)$ and $\subgrad_{\SI}(p)$ are local super-gradients of $\CI$ and $\SI$ and the vectors $\subgrad_{\UIy}(p)$ and $\subgrad_{\UIz}(p)$ are local sub-gradients of $\UIy$ and $\UIz$ in the point $p.$ From this, the statements in this Corollary follow. 
\end{proof}

\section{Application I: Extractable Shared Information}
Let $\RV{S},\RV{Y},\RV{Z}$ are random variables with joint probability distribution~$p$, and denote by $S,Y,Z$ the ranges, respectively, of $\RV{S},\RV{Y},\RV{Z}$.

For a set~$R$ and a $m\in\NN$, a \textit{stochastic $([m]\times R)$-matrix} is a matrix~$\Pi$ with~$m$ rows (indexed $1,\dots,m$ as usual) and columns indexed by the elements of~$R$, whose entries are nonnegative reals such that $\Pi_{*,s}=1$.
Let~$p$ be a probability distribution on $S\times Y \times Z$, and~$\Pi$ be a stochastic $([m]\times S)$-matrix.  Then we define the probabilty distribution $\Pi(p)$ as follows:
\begin{equation*}
  \Pi(p)_{t,y,z} := \sum_{s\in\Rg_0(p)} \Pi_{t,s} p_{s,y,z}, \text{ for all $t\in[m]$ and $ (y,z) \in Y\times Z$.}
\end{equation*}

Rauh et al.~\cite{rauh2017extractable} define two ``extractable'' versions of shared information.  Let $\RV{S},\RV{Y},\RV{Z}$ be random variables with distribution~$p \in \DDelta^{S\times Y\times Z}$.  The \textit{extractable shared information} of $\RV{S},\RV{Y},\RV{Z}$ is defined as
\begin{equation}\label{eq:rbojb-defof-ext}
  \SIext(p) := \sup_{f} \SI(f(\RV{S});\RV{Y},\RV{Z})
\end{equation}
where the supremum is taken over all functions $f\colon S\to T$, where $S$ is the range of~$\RV{S}$ and~$T$ is an arbitrary finite set.  The \textit{probabilistically extractable shared information} is defined as
\begin{equation}\label{eq:rbojb-defof-prext}
  \SIprext(p) := \sup_{\RV{T}} \SI(\RV{T};\RV{Y},\RV{Z})
\end{equation}
where the supremum is taken over all random variables $\RV{T}$ (with finite range) which are conditionally independent of~$\RV{Y},\RV{Z}$ given~$\RV{S}$.

It is straightforward that the extractable shared information of~$p$ is the value of the following optimization problem:
\begin{subequations}\label{eq:def-ext}
  \begin{align}
    \text{With $m := \abs{S}$}:                                            \notag\\
    \SIext(p) := &\max \SI(\Pi(p))                                                 \\
    \text{over }                                                                %
    &\Pi \in \RR^{[m] \times S}                                           \notag\\
    \text{subject to }&                                                   \notag\\
    \label{eq:def-ext:sum1}
    &\Pi_{*,s} = 1      \qquad\text{for all $s\in S$}                     \\
    \label{eq:def-ext:nneg}
    &\Pi_{t,s} \ge 0    \qquad\text{for all $(t,s)\in[m]\times S$}        \\
    \label{eq:def-ext:int}
    &\Pi_{t,s} \in \ZZ  \qquad\text{for all $(t,s)\in[m]\times S$.}
  \end{align}
\end{subequations}
To see why this is the same as the definition~\eqref{eq:rbojb-defof-ext}, given in~\cite{rauh2017extractable}, let us take random variables $\RV{S},\RV{Y},\RV{Z}$ with distribution~$p$.  The \textit{integrality constraints}~\eqref{eq:def-ext:int} --- together with the nonnevativity inequalities~\eqref{eq:def-ext:nneg} and the equation --- have precisely the effect of ensuring that for every~$s$ in the range of~$\RV{S}$ there exists a unique~$t\in[m]$ with $\Pi_{t,s}=1$.  In other words, $\Pi$ defines a mapping from $\Rg\RV{S}$ to $[m]$.  Since $m$ is the size of the range of~$\RV{S}$, the optimization problem~\eqref{eq:def-ext} simply optimizes over all functions defined on the range of~$\RV{S}$, which is exactly~\eqref{eq:rbojb-defof-ext}.

Similarly, the probabilistically extractable shared information is the value of the following optimization problem:
\begin{subequations}\label{eq:def-prext}
  \begin{align}
    {}&\sup \SI(\Pi(p))                                             \\
    \text{over }                                                                %
    \label{eq:def-prext:sup_m}
    m \ge \abs{S}                                                               \\
    &\Pi \in \RR^{[m] \times S}                                           \notag\\
    \text{subject to }&                                                   \notag\\
    \label{eq:def-prext:sum1}
    &\Pi_{*,s} = 1      \qquad\text{for all $s\in S$}                           \\
    \label{eq:def-prext:nneg}
    &\Pi_{t,s} \ge 0    \qquad\text{for all $(t,s)\in[m]\times S$}.             
  \end{align}
\end{subequations}
To see why this is equivalent to the definition~\eqref{eq:rbojb-defof-prext}, given in~\cite{rauh2017extractable}, consider the relation
\begin{equation}\label{eq:rel-between-PI-and-T}
  \Pi_{t,s} = \Prb( \RV{T}=t \mid \RV{S}=s ).
\end{equation}
Given $\Pi$, it defines a random variable~$\RV{T}$ which is conditionally independent of~$\RV{X}_1,\dots,\RV{X}_k$ given~$\RV{S}$, such that $\Pi(p)$ is the distribution of $(\RV{T},\RV{X}_1,\dots,\RV{X}_k)$.
On the other hand, given a random variable~$\RV{T}$ conditionally independent of~$\RV{X}_1,\dots,\RV{X}_k$ given~$\RV{S}$, setting $m := \max\Rg_0$, relation~\eqref{eq:rel-between-PI-and-T} defines a $\Pi$ such that $\Pi(p)$ is the distribution of $(\RV{T},\RV{X}_1,\dots,\RV{X}_k)$.
We invite the reader to check these claims --- %\todo[Abed:Careful about conditioning on events with prb 0]
or read the detailed proof in~\cite[Lemma~5.2.1]{makkeh:phd:2018}.%\TODO[WHERE??].

There are two significant differences between the \eqref{eq:def-ext} and~\eqref{eq:def-prext}.  Firstly, it lacks the integrality constraints, making it a continuous optimization problem.  Secondly, the dimension, $m$, is a variable, making the optimization problem infinite dimensional (as observed in~\cite{rauh2017extractable}), and thus basically\footnote{Approximation through is thinkable.} intractable from an algorithmic point of view.  (The lower bound $m \ge S$ is redundant, see Lemma~\ref{lem:relations-between-extrmeasures} below).

The following optimization problem, however, is a standard continuous optimization problem to which we can apply our results: For a fixed value of~$m \in \NN$, let us define
\begin{subequations}\label{eq:def-myext}
  \begin{align}
    \mySI_m(p) := &\max \SI(\Pi(p))                                             \\
    \text{over }                                                                %
    &\Pi \in \RR^{m \times S}                                             \notag\\
    \text{subject to }&                                                   \notag\\
    \label{eq:def-myext:sum1}
    &\Pi_{*,s} = 1      \qquad\text{for all $s\in S$}                           \\
    \label{eq:def-myext:nneg}
    &\Pi_{t,s} \ge 0    \qquad\text{for all $(t,s)\in[m]\times S$}.             
  \end{align}
\end{subequations}

The following lemma is quite obvious (see \cite[Lemma~5.2.2]{makkeh:phd:2018}%\TODO[WHERE??] 
for a detailed proof).

\begin{lemma}\label{lem:relations-between-extrmeasures}
  The sequence $m\mapsto \mySI(m)$ is non-decreasing and for every fixed $m_0\ge \abs{S}$,
  \begin{equation*}
    \SIext(p) \le \mySI_{m_0}(p) \le \sup_{m\ge0} \mySI_m(p) = \SIprext(p).
  \end{equation*}
\end{lemma}

\subsection*{Acknowledgements}
This research was supported by the Estonian Research Council, ETAG (\textit{Eesti Teadusagentuur}), through PUT Exploratory Grant \#620.  We also gratefully acknowledge funding by the European Regional Development Fund through the Estonian Center of Excellence in Computer Science, EXCS.

%%%%%%%%%%%%%%%%%%%%%%%%%%%%%%%%%%%%%%%%%%%%%%%%%%%%%%%%%%%%%%%%%%%%%%%%%%%%%%%%%%%%%%%%%%%%%%%%%%%%%%%%%%%%%%%%
\bibliographystyle{plain}

%%%%%%%%%%%%%%%%%%%%%%%%%%%%%%%%%%%%%%%%%%%%%%%%%%%%%%%%%%%%%%%%%%%%%%%%%%%%%%%%%%%%%%%%%%%%%%%%%%%%%%%%%%%%%%%%
%\appendix
%\section*{APPENDIX}
%\section{Blah blah}
\end{document}